\documentclass[12pt]{article}

\usepackage{graphicx}
\usepackage{latexsym,amssymb}

\usepackage{amsthm}
\usepackage{indentfirst}
\usepackage{amsmath}
\usepackage{hyperref}

\newtheorem{theoremalph}{Theorem}

\newtheorem{Theorem}{Theorem}[section]
\newtheorem*{Theorem A}{Theorem A}

\newtheorem{Definition}[Theorem]{Definition}
\newtheorem{Proposition}[Theorem]{Proposition}
\newtheorem{Lemma}[Theorem]{Lemma}

\newtheorem*{Claim}{Claim}
\newtheorem*{Acknowledgements}{Acknowledgements}

 \def\NN{{\mathbb N}}

\def\dim{\operatorname{dim}}

\begin{document}

\title{On the finiteness of uniform sinks \footnote{2000 Mathematics Subject Classification. 37D30}
}

\author{  Dawei Yang \footnote{D. Yang was partially supported by NSFC 11271152, Ministry of Education of P. R. China 20100061120098}\and Yong Zhang
\footnote{Y. Zhang war partially supported by NSFC 11271152}}

\date{}

\maketitle


\
\begin{abstract}
We study the finiteness of uniform sinks  for flow. Precisely, we prove that, for $\alpha>0$  $T>0$,
if a vector field $X$ has only hyperbolic singularities or sectionally dissipative singularities, then $X$ can have only finitely many $(\alpha,T)$-uniform sinks.
This is a generalized version of a theorem of Liao \cite{Lia89}.
\end{abstract}

\section{Introduction}

In this work, we give a generalized version of a theorem of Liao \cite{Lia89}. It could be seen as an extension of
the remarkable Pliss' theorem \cite{Pli72} in the setting of singular flows.

Let $M$ be a compact smooth Riemannian manifold and $X$ be a smooth vector field on $M$. We know that $X$ will generate a smooth flow $\phi_t$. If $X(\sigma)=0$, $\sigma$ is called a \emph{singularity} of $X$. If $\phi_t(p)=p$ for some $t>0$ and $X(p)\neq 0$, $p$ is called a \emph{periodic point}. We use ${\rm Sing}(X)$ and ${\rm Per}(X)$ to denote the sets of singularities and periodic points.

The flow $\Phi_t=d \phi_t:TM\to TM$ is called the \emph{tangent flow}. Note that every periodic orbit has at least one zero Lyapunov exponent w.r.t. $\Phi_t$.
To understand the dynamics in a small neighborhood of a periodic orbit, Poincar\'e used the Poincar\'e return map: for any point in the periodic orbit, one takes a cross section at that point, then the flow defines a local diffeomorphism in a small neighborhood of the cross section. The dynamics of the flow in a small neighborhood of the periodic orbit can be understood by the dynamics of the diffeomorphism.

By extending this idea to the general non-periodic case, for any regular point $x$ and any $t\in\mathbb R$, one considers local normal cross sections at $x$ and $\phi_t(x)$, then the flow gives a local diffeomorphism between these two cross sections. Its linearization is the \emph{linear Poincar\'e flow} $\psi_t$, which is defined as the following:
given a regular point $x\in M$, consider a vector $v$  in the orthogonal complement of $X(x)$,
one defines
$$\psi_t(v)=\Phi_t(v)- {<\Phi_t(v),X(\phi_t(x))>\over |X(\phi_t(x))|^2} X(\phi_t(x)).$$
Note that $\psi_t$ cannot be defined on the singularities.

Given $\alpha>0$ and $T>0$, a periodic orbit $\Gamma$ is called an \emph{($\alpha,T$)-uniform sink} if there are $m\in\NN$ and times $0=t_0<t_1<t_2\cdots<t_n=m\pi(\Gamma)$ ($\pi(\Gamma)$ is the period of $\Gamma$) satisfying $t_i-t_{i-1}\le T$ for any $1\le i\le n$ such that for any $x\in\Gamma$, one has
$$\prod_{i=1}^n\|\psi_{t_{i}-t_{i-1}}(\phi_{t_{i-1}}(x))\|\le {\rm e}^{-\alpha m\pi(\Gamma)}.$$

A singularity $\sigma$ is called \emph{sectionally dissipative} if the following is true: when we list all the eigenvalues of $DX(\sigma)$ as $\{\lambda_1,\lambda_2,\cdots,\lambda_d\}$, we have ${\rm Re}(\lambda_i)+{\rm Re}(\lambda_j)\le0$ for any $1\le i<j\le d$. Here $d$ is the dimension
of the manifold $M$.

Notice that for a sectionally dissipative non-hyperbolic singularity, the maximal of the real parts of its eigenvalues should be zero.

\begin{theoremalph}\label{Thm:main}
Let $\alpha>0$  $T>0$. If a vector field $X$ has only hyperbolic singularities or sectionally dissipative singularities, then $X$ can have only finitely many $(\alpha,T)$-uniform sinks.
\end{theoremalph}

Similar results for diffeomorphisms or non-singular flows were got by Pliss \cite{Pli72}. If $X$ has no singularities and $X$ has infinitely many $(\alpha,T)$-uniform sinks $\{\gamma_n\}$, then by using Pliss Lemma, for any $\gamma_n$, there is a point $x_n\in\gamma_n$ such that $x_n$ has its stable manifold of uniform size which is independent with $n$. From the fact that $M$ has finite volume, we can get a contradiction.

Liao \cite{Lia89} proved Theorem~\ref{Thm:main} with an additional assumption: $X$ is a \emph{star} vector field. As $X$ is star, every singularity of $X$ is hyperbolic. If every singularity of $X$ is hyperbolic and $X$ has infinitely many $(\alpha,T)$-uniform sinks ${\gamma_n}$, then by his estimation, for each $n$, there is a point $x_n\in \gamma_n$ such that $x_n$ has its stable manifold with some uniformity (after the rescaling of the flow). So if $x_n$ is far away from singularities, we can get a contradiction. Thus we can assume that $\lim_{n\to\infty}x_n=\sigma$ for some singularity $\sigma$. Then Liao proved that $\sigma$ has a strong unstable one-dimensional manifold and for $n$ large enough, the basin of $x_n$ intersects the strong unstable manifold. Since a one-dimensional strong unstable manifold contains only two orbits, we can get a contradiction.

But in Liao's argument, if $\sigma$ is not hyperbolic, we will encounter two difficulties:
\begin{itemize}

\item $\sigma$ may not have a strong unstable one-dimensional manifold. To solve this difficulty, we need to analysis the dynamics on the central manifold of $\sigma$.

\item When $\sigma$ is hyperbolic, then we know that for $x$ close to $\sigma$,
$$\frac{|X(x)|}{d(x,\sigma)}$$
is uniformly bounded. Thus, Liao needed to consider a cone along the unstable direction.
When $\sigma$ is non-hyperbolic, we need to consider some general cone-like region along the central manifold of $\sigma$.

\end{itemize}

\section{Preliminaries}

\subsection{The rescaled linear Poincar\'e flow and the stable manifold theorem}

Let $X$ be a $C^1$ vector field. For each regular point $x$, one defines its normal space ${\cal N}_x$ to be
${\cal N}_x=\{v\in T_x M:~<v,X(x)>=0\}.$ Denote by
$${\cal N}=\bigcup_{x\in M\setminus{\rm Sing}(X)}{\cal N}_x.$$
${\cal N}$ is called the \emph{normal bundle} of $X$. Note that $M\setminus {\rm Sing}(X)$ may be not
 compact. Thus, ${\cal N}$ may be defined on some non-compact set. We notice that $\psi_t$ is defined on ${\cal N}$. Given any regular point $x\in M$, $t\in{\mathbb R}$ and any $v\in{\cal N}_x$, we define
$$\psi_t^*(v)={\psi_t(v)\over \|\Phi_t|_{\left< X(x)\right>}\|}=\frac{|X(x)|}{|X(\phi_t(x))|}\psi_t(v).$$
$\psi_t^*$ is called the \emph{rescaled linear Poincar\'e flow} w.r.t. $X$.

\begin{Definition}

Let $C>0$, $\eta>0$ and $T>0$. A regular point $x\in M$ is called \emph{$(C,\eta,T)$-$\psi^*$-contracted}, if there is a sequence of times $0=t_0<t_1<\cdots<t_n<\cdots$ such that
\begin{itemize}

\item $t_i-t_{i-1}\le T$ and $\lim_{n\to\infty}t_n=\infty$.

\item $\prod_{i=1}^{n}\|\psi^*_{t_i-t_{i-1}}(\phi_{t_{i-1}}(x))\|\le C{\rm e}^{-\eta t_n}$ for any $n\ge 1$.

\end{itemize}

One says that a regular point $x\in M$ is  \emph{$(C,\eta,T)$-expanded} if it is $(C,\eta,T)$-contracted for $-X$.
\end{Definition}

For a normed vector space $V$ and $r>0$, denote by
$$V(r)=\{v\in V:~|v|\le r\}.$$

For any regular point $x\in M$, we define the local normal manifold $N_x(\beta)=\exp_x({\cal N}_x(\beta))$ ($\beta>0$).
 The flow $\phi_t$ defines a local diffeomorphim from a small neighborhood of $N_x(\beta)$ to $N_{\phi_t(x)}(\beta)$, which is denoted by $P_{x,\phi_t(x)}$, and which is called
 the \emph{sectional Poincar\'e map}.

Liao \cite{Lia89} had the following estimations on the size of stable manifolds. One can see \cite[Section 2]{GaY12} for a geometric proof.

\begin{Lemma}\label{Lem:uniformestimate}
Let  $X$ be a $C^1$ vector field on $M$. Given $C>0$,  $\eta>0$ and $T>0$,  there is $\delta=\delta(C,\eta,T)>0$ such that
 for any $(C, \eta,T)$-$\psi^*$-contracted point $x$, one has
$N_x(\delta|X(x)|)$ is in the domain of the sectional Poincar\'e map $P_{x,\phi_t(x)}$ for any $t\ge 0$, and
$$\lim_{t\to\infty}{\rm diam}(P_{x,\phi_t(x)}({N}_x(\delta|X(x)|)))=0.$$

\end{Lemma}

It follows that ${N}_x(\delta|X(x)|)$ is in the stable set of $x$ after a reparametrization. Although we don't want to give the proof of Lemma~\ref{Lem:uniformestimate} again, we would like to give some idea about the proof. First we can consider the fibered map ${\cal P}_{t,\phi_t(x)}:~{\cal N}_x(\beta)\to {\cal N}_{\phi_t(x)}(\beta)$, which is defined by ${\cal P}_{x,\phi_t(x)}= \exp^{-1}_{\phi_t(x)} \circ P_{x,\phi_t(x)} \circ\exp_x$, whose dynamics are conjugate to $P_{t,\phi_t(x)}$. The linearization of ${\cal P}_{x,\phi_t(x)}$ is the linear Poincar\'e flow $\psi_t$. But since $X$ may contain singularities, the linearized neighborhood is not uniform. Then we define the rescaling of ${\cal P}_{x,\phi_t(x)}$ by
$${\cal P}^*_{x,\phi_t(x)}(v)=\frac{{\cal P}_{x,\phi_t(x)}(|X(x)|v)}{|X(\phi_t(x))|}.$$
The linearization of ${\cal P}^*_{x,\phi_t(x)}$ is $\psi_t^*$ and the linearized neighborhood is uniform; by a careful calculation, $D {\cal P}^*_{x,\phi_t(x)}$ also have some uniform continuity properties (See \cite{GaY12}). By our assumption, if $x$ is $(C, \eta,T)$-$\psi^*$-contracted, then $x$ has its stable manifold of uniform size w.r.t. ${\cal P}^*$: the proof follows from the classical case of diffeomorphisms (see \cite[Corollary 3.3]{PuS00} for instance).

\subsection{A lemma of Pliss type}

\begin{Lemma}\label{Lem:infinitePliss}

Given $C>0$ and $0<\lambda_1<\lambda_2<1$, there is $N=N(C,\lambda_1,\lambda_2)\in\NN$ such that: for any sequence of numbers $\{a_n\}$ satisfying the following properties:
$$\sum_{i=1}^n a_i\le C+n\lambda_1,~~~\forall n\in\NN,$$

then there is $L\le N$ such that
$$\sum_{i=1}^n a_{L+i}\le n\lambda_2,~~~\forall n\in\NN.$$

\end{Lemma}

\begin{proof}
We choose $N$ such that $C+N\lambda_1 < N\lambda_2$. Given any sequence of numbers $\{a_n\}$ satisfying $\sum_{i=1}^n a_i\le C+n\lambda_1,~\forall n\in\NN$, there is $m\in\NN$ such that for any $n\in\NN$, one has $$\sum_{i=1}^n a_{m+i}\le n\lambda_2.$$ This is because otherwise, for each $j$, there is $n_j$ such that
$\sum_{i=1}^{n_j} a_{j+i}\ge n_j\lambda_2.$ Thus $\limsup_{n\to\infty}\frac{1}{n} \sum_{i=1}^n a_i\ge\lambda_2$. This fact contradicts to the assumption.

Now we will show the existence of $L$, which is required to be less than $N$. If not, there is a sequence of numbers $\{a_n\}$ satisfying
$$\sum_{i=1}^n a_i\le C+n\lambda_1,~~~\forall n\in\NN,$$
but for any $m$ satisfying
$$\sum_{i=1}^n a_{m+i}\le n\lambda_2,~~~\forall n\in\NN$$
one has $m>N$. We take a minimal $m$ with the above property. This implies $a_{m-1}>\lambda_2$. Inductively, we can have that
\begin{Claim}
For any $j\le m-1$, $\sum_{k=0}^{m-1-j}a_{j+k}> (m-j)\lambda_2$.

\end{Claim}
\begin{proof}
We know that it is true for $j=m-1$. If it is already true for $j,j+1,\cdots,m-1$, and it is not true for $j-1$, then we will have
$$\sum_{k=0}^{m-1-(j-1)}a_{j-1+k}\le (m-j+1)\lambda_2,$$
$$\sum_{k=0}^{m-1-\ell}a_{\ell+k}> (m-\ell)\lambda_2,~~~\forall \ell\in\{j,j+1,\cdots,m-1\}.$$

This will imply that for any $0\le \ell\le m-j-1$, one has
$$\sum_{i=0}^{m-j}a_{j-1+i}\le (m-j+1)\lambda_2.$$
By the definition of $m$, we know that $j-1$ also have the same property of $m$. This will contradict to the minimality of $m$.

\end{proof}

As a consequence of the above claim, one has $\sum_{i=1}^{m-1}a_i>(m-1)\lambda_2\ge N\lambda_2\ge C+N\lambda_1$. This gives a contradiction to the assumption of $\{a_n\}$.

\end{proof}

\section{The splitting of the singularity}

Recall the sphere bundle $S^1(M)$ which consists of all unit  tangent vectors of the tangent bundle $TM$:
$$ S^1(M)=\{v\in TM, |v|=1\}.$$
Thus, $\Phi_t$ induces a continuous flow $\Phi_t^1$ on $S^1(M)$,   where
$$  \Phi_t^1(v)= {\Phi_t(v)\over |\Phi_t(v)|} \quad (v\in S^1 (M)). $$

Define the following frame by
$${\cal F}_2(M)=\bigcup_{x\in M}\{(u,v):~u,v\in T_x M,~u\neq 0,~u\perp v\}.$$
and the normalized frame by
$${\cal F}^\#_2(M)=\bigcup_{x\in M}\{(u,v):~u,v\in T_x M,~|u|=1,~u\perp v\}.$$

Define the flows $\chi_t$ on ${\cal F}_2(M)$ and $\chi_t^\#$ on ${\cal F}^\#_2(M)$ by the following way:

$$\chi_t(u,v)=\{\Phi_t(u),~\Phi_t(v)-\frac{<\Phi_t(u),\Phi_t(v)>}{|\Phi_t(v)|^2} \Phi_t(u)\}$$

$$\chi_t^\#(u,v)=\{\frac{\Phi_t(u)}{|\Phi_t(u)|},~\Phi_t(v)-\frac{<\Phi_t(u),\Phi_t(v)>}{|\Phi_t(v)|^2} \Phi_t(u)\}$$

Note that ${\cal F}^\#_2(M)$ is a complete metric space and $\chi_t^\#$ is a continuous flow.

\begin{Definition}
For any set $\Lambda\subset S^1(M)$, we say that $\chi_t$ is \emph{dominated} on $\Lambda$ if there are $C>0$, $\lambda>0$ such that for any $t>0$ and any $u\in \Lambda$, one has
$$\frac{\|{\rm proj}_2\chi_t (u,\cdot)\|}{\|{\rm proj}_1 \chi_t (u, \cdot)\|}=
\frac{\|{\rm proj}_2\chi_t(u,\cdot)\|}{|\Phi_t(u)|}\le C{\rm e}^{-\lambda t}.$$
\end{Definition}

\begin{Lemma}
If $\Lambda\subset S^1(M)$ is dominated w.r.t. $\chi_t$, then its closure $\overline{\Lambda}$
is also dominated w.r.t. $\chi_t$.
\end{Lemma}
\proof By taking a limit, we know that this lemma is true. \qed

\begin{Lemma}
For every singularity $\sigma$, given $t\in \mathbb{R}$, $\Phi_t (u)$ is $C^{\infty}$ w.r.t. $u\in T_{\sigma} M$ (although $X$ is only $C^1$).
\end{Lemma}

\proof This is true because that $\Phi_t$ is linear w.r.t. $u\in T_{\sigma}M$.\qed

\begin{Lemma}
For any $u \in S_{\sigma} ^1 M$ and $t\in \mathbb{R}$, one has
$$ D_u \Phi_t^1 = \frac{{\rm proj}_2 \chi_t (u, \cdot)}{ |\Phi_t (u)|}.$$
\end{Lemma}

\proof  For each $u \in S_\sigma^1 M$, $T_u S^1_\sigma M$ can be identical with $\mathcal{N}_u$, where $\mathcal{N}_u=\{v\in T_\sigma M:
v\perp u\}$. $D_u \Phi_t^1$ ia s map from $\mathcal{N}_u$ to $\mathcal{N}_{\Phi_t^1(u)}$, which can be got
by the following way: for each $v\in \mathcal{N}_u$, we need to project $\Phi_t(v)$ to $\mathcal{N} _{\Phi_t^1 (u)}$. This process gives ${\rm proj}_2 \chi_t (u, \cdot)$. Since $\Phi_t(v)$ may not be unit, we need to do a scaling. This ends the proof. \qed

\begin{Proposition}\label{Pro:splittingsingularity}
Let $C>0$, $\eta>0$ and $T>0$. For a singularity $\sigma$, if there is a sequence of
$(C,\eta,T)$-$\psi^*$-contracted points \{$x_n$\} satisfying $\lim_{n\to\infty}x_n=
\sigma$, then $\sigma$ admits a dominated splitting $E\oplus F$ w.r.t. the tangent flow $\Phi_t$ and $\dim F=1$. Moreover,
any accumulation point of $\{X(x_n)/|X(x_n)|\}$ in $S^1 M$ is not in $E$.

\end{Proposition}

\proof By our assumptions, $\{X(x_n)/|X(x_n)|\}$ is dominated  w.r.t. $\chi_t$. So any accumulation point
$u$ of $X(x_n)/ |X(x_n)|$ is dominated w.r.t. $\chi_t$. Given $T>0$, let $f= \Phi_T^1$.
By the assumptions, there are $C>0$ and $\lambda \in (0,1)$ such that for any $n\in \mathbb{N}$, one has
$$  \prod_{i=0}^{n-1} \|Df (f^i (u))\|  \leq C\lambda^n.$$
Fix some $\lambda_1\in (\lambda ,1)$. By Lemma \ref{Lem:infinitePliss},
there exists an infinite sequence $\{n_i\}_{i\in \mathbb{N}}$ such that
$$ \prod_{j=0}^{n-1} \|Df (f^j (f^{n_i} (u)))\| \leq \lambda_1^n, \forall \ n \in \mathbb{N}.$$
Choose $n_i$, $n_j$ such that $n_j>n_i$, $f^{n_i} (u)$ and $f^{n_j} (u)$ are close enough. Thus,
there is $\delta>0$ such that $f^{n_j - n_i}$  is a contracting map on $B(f^{n_i} (u), \delta)$.
This implies that $f$ has  a periodic point. Thus, there is $T^{'} >0$ and $u' \in S_\sigma^1 M$ such
that $\Phi_{T'}^1 (u') = u'$ and $u'$ is dominated w.r.t.  $\chi_t$.

Now we have that $\Phi_{T'}$  has the following form w.r.t. $\langle u'\rangle  \oplus \mathcal{N}_{u'}$
\[\displaystyle\left(\begin{array}{cc}
    \Phi_{T'}  (u')&  0 \\
    A &  {\rm proj}_2 \chi_{T'} (u', \cdot)
  \end{array}
\right).\]
This implies that $\Phi_{T'}$ has a unique largest eigenvalue. From these facts, one can
get the dominated splitting on $T_\sigma M$.

Once we know that we have the dominated splitting  $T_\sigma M = E \oplus F$ with ${\rm dim} F=1$, we
know that any point which is dominated w.r.t. $\chi_t$ cannot be in $E$. Thus, every accumulation point
of $ X(x_n)/|X(x_n)|$ is not in $E$. \qed

 \section{The intersection of local invariant manifolds}

\subsection{The central manifold at $\sigma$}

In this subsection, we assume that $\sigma$ is a singularity and $T_\sigma M$ admits a dominated splitting $E\oplus F$ w.r.t. the tangent flow $\Phi_t$ and $\dim F=1$. We will talk about the local dynamics around $\sigma$.

Since $T_\sigma M=E\oplus F$ with $\dim F=1$ is a dominated splitting of $\Phi_t$, by the plaque family theorem of \cite[Theorem 5.5]{HPS77}, there is a local embedded one-dimensional manifold $W^F(\sigma)$ which is centered at $\sigma$ and tangent to $F$ at $\sigma$; moreover, it is locally invariant in the following sense: there is a $C^1$ map $g:F\to E$ in $T_\sigma M$ satisfying $g(0)=0$ and $Dg(0)=0$, if we denote  by $W^F_r(\sigma)=\exp_x(g(-r,r))$, we have that for any $\varepsilon>0$, there is $\delta>0$ such that
$$\phi_{[-1,1]}(W^F_\delta(\sigma))\subset W^F_\varepsilon(\sigma).$$

We notice that $W^F(\sigma)$ may not be unique in general. $W^F(\sigma)$ has two seperatrix, which are denoted by $W^{F,+}(\sigma)$ and $W^{F,-}(\sigma)$. We discuss $W^{F,+}$ only since they are in a symmetric position. If $E$ is uniformly contracting, then $W^{ss}_{loc}(\sigma)$ separates a small neighborhood of $\sigma$ into two parts: the upper part which contains $W^{F,+} (\sigma)$, and the lower part which contains $W^{F,-}(\sigma)$. If $W^{F,+}(\sigma)$ is Lyapunov stable in the following sense: for any $\varepsilon>0$, there is $\delta>0$ such that $\varphi_t(W_\delta^{F,+}(\sigma))\subset W_\varepsilon^{F,+}(\sigma)$ for any $t>0$, then the upper part will be foliated by locally strong stable foliations.

\begin{Lemma}\label{Lem:uniquelyintegrable}

If $W^{F,+}(\sigma)$ is contained in the unstable manifold of $\sigma$, then it is uniquely defined.

\end{Lemma}

\begin{proof}
Under the assumptions, $F$ cannot be a contracting bundle. If $F$ is an expanding bundle, then the conclusion follows from the classical theorem about the existence of unstable manifolds.

So we can assume that $F$ is non-hyperbolic. In this case, since we have the dominated splitting $T_\sigma M=E\oplus F$, we have that $E=E^s$ is contracting.
We can extend the bundle $E$ and $F$ in a small neighborhood $U$ of $\sigma$, which are still denoted by $E$ and $F$. For $\alpha>0$, for any point $x\in U$, we define the cone field ${\cal C}^E_\alpha(x)\subset T_x M$ :
$${\cal C}^E_\alpha(x)=\{v\in T_x M:~v=v_E+v_F,~v_E\in E(x),~v_F\in F(x),~\textrm{and}~|v_F|\le\alpha |v_E|\}.$$

By reducing $U$ if necessary, one has
\begin{Claim}
There are $\alpha>0$,  $T>0$ and $\varepsilon>0$ such that
\begin{enumerate}

\item For any $x\in U$, if $\phi_{[-T,0]}(x)\in U$, then $\Phi_{-T}(x)({\cal C}^E_\alpha(x))\subset {\cal C}^E_\alpha(\phi_{-T}(x))$.

\item For any $x,y\in U$ such that $d(x,y)<\varepsilon$, if $\phi_{[-T,0]}(x)\in U$ and $\phi_{[-T,0]}(y)\in U$, and if $\exp_x^{-1}(y)\in {\cal C}^E_\alpha(x)$, then $\exp_{\phi_{-T}(x)}^{-1}(y)\in {\cal C}^E_\alpha(\phi_{-T}(x))$, and moreover $d(\phi_{-T}(x),\phi_{-T}(y))>2d(x,y)$.

\end{enumerate}

\end{Claim}

\begin{proof}[Proof of the Claim]
First at the singularity $\sigma$, there is $T>0$ and $\alpha>0$ such that

\begin{itemize}

\item $\Phi_{-T}({\cal C}^E_\alpha(\sigma))\subset {\cal C}^E_{\alpha/2}(\phi_{-T}(\sigma))$,

\item for any unit vector $v\in {\cal C}^E_\alpha(\sigma)$, $|\Phi_{-T}(v)|>2$.

\end{itemize}

So by the continuity of $E$, by reducing $U$ if necessary, we have it is also true for any $x\in U$. Thus Item 1 of the claim is true.
Next, since the linearization of $\phi_{-T}$ is uniformly continuous, we have that the existence of $\varepsilon$ such that Item 2 of the claim is true.

\end{proof}

Now we continue the proof of Lemma \ref{Lem:uniquelyintegrable}. We will prove it by absurd.
Suppose  that there exists another central manifold $\widetilde{W}^{F,+}$.
This central manifold $\widetilde{W}^{F,+}$ could not be Lyapunov stable, otherwise the upper part of a small neighborhood of $\sigma$ will be foliated by
a strong stable foliation. As a corollary, the original central manifold $W^{F,+}$ is also Lyapunov stable. This is a contradiction.
Now we choose two points $x\in W^{F, +}$ and $\tilde{x} \in \widetilde{W}^{F,+}$ such that $d(x,\sigma)\ll \varepsilon $,
$d(\tilde{x}, \sigma)\ll \varepsilon$ and a curve $\gamma$ connecting
$x$ and $\tilde{x}$ verifying $\exp_x^{-1}(\gamma) \subset {\cal C}^E_\alpha (x)$. The negative iterations $\gamma$ of $\phi_{-nT}$ ($n=1,2,\cdots$) will have the property:
$d(\phi_{-nT} (x), \phi_{-nT} (\tilde{x}))>2^nd(x, \tilde{x})\geq \varepsilon $ for $n$ large enough. This contradicts to the triangle inequality.

\end{proof}

\subsection{The cone-like region }
For $\alpha>0$, one considers the cone
$${\cal C}^F_\alpha=\{v\in T_\sigma M:~v=v_E+v_F,~v_E\in E,~v_F\in F,~\textrm{and}~|v_E|\le\alpha |v_F|\}.$$
and  considers $C^F_\alpha=\exp_\sigma ({\cal C}^F_\alpha\cap T_\sigma M(\beta))$.

We extend $E$ and $F$ continuously to a  small  neighborhood $U$ of $\sigma$. Let $\alpha>0$ and $\beta>0$ be given.
Now we define a cone-like region ${\cal D}_\alpha^ F (\beta)$,
\begin{eqnarray*}
{\cal D}_\alpha^ F (\beta)&=&\{x\in M: d(x,\sigma)<\beta, X(x)=X_F (x)+X_E(x), X_F (x)\in F(x),\\
               &~& X_E(x)\in E(x), |X_E (x)|<\alpha |X_F (x)| \}.
\end{eqnarray*}

\begin{Lemma}\label{Lem:smallscale}
Assume that $DX(\sigma)|_E$ is non-singular. For any $\delta>0$, there are $\alpha>0$ and $\beta>0$ such that for any $y\in {\cal D}_\alpha^ F (\beta)$, one has $\exp_{y}({\cal N}_{y}(\delta|X(y)|))\cap W^F(\sigma)\neq\emptyset$.
\end{Lemma}
\begin{proof} Without loss of generality, we can identify a small neighborhood $U$ to be an open set in ${\mathbb R}^d$, and $\sigma$ to be the origin
point  $0$. $E$ and $F$ are two subspaces perpendicular to each other. Thus $W^F(\sigma)$ is a $C^1$ curve tangent to $F$ at the origin.
Denote by $x=(x_E, x_F)$ for each point $x$ in the small neighborhood. We assume that the flow generated by the vector filed in a small neighborhood of $\sigma$ is the solution of the following differential equation:

\[\begin{array}{ll}\
&\displaystyle\frac{\mathrm{d}x_E}{\mathrm{d}t}=A_E x_E + f_E(x), \\[0.2cm] &\displaystyle\frac{\mathrm{d}x_F}{\mathrm{d}t}=A_F x_F +  f_F(x),
\end{array}\]
where $f_E$ and $f_F$ are higher order terms of $|x|$, and the matrix

$$\begin{pmatrix}
A_E~~~ 0\\
0 ~~~ A_F
\end{pmatrix}$$
can be regarded as $DX(\sigma)$. Under our assumptions, we always know that $A_E$ is non-singular since $E$ is dominated by $F$.

Thus we can find a constant $c>0$ which depends only on $A$ such that for small enough $\beta>0$, if $d(x, \sigma)<\beta$ then $|x_E|\leq c |X_E(x)|$.

Through an argument of basic geometry, if $|X_E (x)|\leq   \alpha |X_F (x)| $, then
the lengths of $E$ component of $N_x(tX(x))$ is greater than
$c_0t |X(x)|$ provided $x$ close to $\sigma$ enough for some uniform constant $c_0=c_0(\alpha)>0$.

Now let $\delta>0$ be given. Choose $\alpha>0$ satisfying $\alpha c/c_0\leq \delta$. We
also choose $\beta>0$ satisfying the conditions above. Let $y\in {\cal D}_\alpha^ F (\beta)$.
 To prove $\exp_{y}({\cal N}_{y}(\delta|X(y)|))\cap W^F(\sigma)\neq\emptyset$, it is enough to show that
 the length of $E$ component of $\exp_{y}({\cal N}_{y}(\delta|X(y)|))$ is greater that $|y_E|$.
 From the first part, we have
 \[|y_E| \leq c|X_E(y)|\leq c \alpha |X_F(y)|\leq c \alpha |X(y)|.\]
 From the second part,  the size of $E$ component of $N_y(\delta X(y))$ is greater
 that $c_0 \delta |X(y)|$, which is greater than $c \alpha |X(y)|$ and hence greater than $|E_y|$.
 \end{proof}

\begin{Lemma}\label{Lem:insidecone}
 Let $\eta>0$ and $T>0$. Assume that $\sigma$ is a singularity and there is a sequence of $(1,\eta,T)$-$\psi^*$-contracted points $\{x_n\}$ such that $\lim_{n\to\infty}x_n=\sigma$. For any $\alpha>0$, there is $L=L(\alpha)>0$ such that for any $L'>L$ and any $\beta>0$,
 there exists an integer $N= N(\alpha, \beta, L')$ such that for any $n\geq N$,
  $\phi_{[L,L']}(x_n)\in {\cal D}_\alpha^ F(\beta)$.
\end{Lemma}
\begin{proof}
 By taking a subsequence if
 necessary, we have that $\lim_{n\to\infty}X(x_n)/|X(x_n)|=v\in T_\sigma M$.
 First by Proposition~\ref{Pro:splittingsingularity}, $\sigma$ admits
 a dominated splitting  $E\oplus F$ with respect to  the tangent flow and $\dim F=1$.
 Given  $\alpha>0$, by the domination and $v\not\in E$,  there exists $L =L(\alpha)$
 such that for any $t>L$,
  $\Phi_t(v)\in{\cal D}^F_\alpha$.

Let $\beta>0$
and $L^{'}>L$ be given. By the fact that $x_n\to \sigma$,
 ${X(x_n) \over |X(x_n)|} \to v$ and the continuity of $\Phi_t$, there is some positive
 integer $N=N(\alpha, \beta, L')$ such that for any  $n\geq N$,
  we have $d(x_n, \sigma)< \beta$,  $\Phi_{[L,L']}^1(X(x_n)/|X(x_n)|)\subset {\cal C}_\alpha^F$.
  The last relation implies that $
  |X_E (x)|<   \alpha |X_F (x)| $ for $x\in \phi_{[L, L'] }(x_n)$ and implies that
  $\phi_{[L,L']}(x_n)\subset {\cal D}_\alpha^ F(\beta)$.

\end{proof}

\section{The final proof of Theorem \ref{Thm:main}}

%
%
%
%
%
%
%
%
%
%

We will prove Theorem~\ref{Thm:main} in this section. Assume that we are under the assumptions of Theorem~\ref{Thm:main}. We will prove it by contradiction. We assume that the conclusion of Theorem~\ref{Thm:main} is not true. That is,  there are infinitely many distinct $(\alpha,T)$-sinks $\{\gamma_n\}$ of $X$. If the period of $\gamma_n$ is bounded, then the limit point of $\{\gamma_n\}$ in the Hausdorff topology should be a non-hyperbolic periodic orbit $\gamma$. Then each periodic orbit close to $\gamma$ cannot be a $(\alpha,T)$-sink. This gives a contradiction.

Hence, we can assume that the period of $\gamma_n$ tends to infinity. By Pliss Lemma \cite{Pli72}, there is $x_n\in\gamma_n$ which is $(1,\eta,T)$-$\psi^*$-contracted for some $\eta>0$. Then by Lemma~\ref{Lem:uniformestimate}, there is $\delta>0$ such that $N_{x_n}(\delta|X(x_n)|)$ is contained in the basin of $\gamma_n$. As a corollary, the ball $B(x_n,\delta|X(x_n)|)$ is also contained in the basin of $\gamma_n$ by reducing $\delta$ if necessary.

Without loss of generality, we can assume that $\{x_n\}$ converges. If the limit point is a regular point, then the basin of $\gamma_n$ will contains a uniform ball for each $n$. This will contradict to the infiniteness. Thus, we can assume that $\lim_{n\to\infty}x_n=\sigma$, where $\sigma$ is a singularity.

By the assumptions, $\sigma$ will be hyperbolic or sectionally dissipative. By Proposition~\ref{Pro:splittingsingularity}, there is a dominated splitting $T_\sigma M=E\oplus F$ w.r.t. the tangent flow, where $\dim F=1$. In any case we will have $DX(\sigma)|_E$ is non-singular.

By the theory of invariant manifolds, we know the existence of $W^F_\varepsilon(\sigma)$.

\paragraph{We consider the case that $F$ is expanded.}
Let $\delta =\delta(1, {\eta/2}, T)$ given by Lemma \ref{Lem:uniformestimate}. By Lemma \ref{Lem:smallscale},
for this $\delta$, there exist $\alpha>0$ and $\beta>0$ such that for any $y\in {\cal D}_\alpha^ F (\beta)$, one has $\exp_{y}({\cal N}_{y}(\delta|X(y)|))\cap W^F(\sigma)\neq\emptyset$.
Lemma \ref{Lem:insidecone} gives the number $L=L(\alpha, \beta)>0$.
$\phi_{[L,L']}(x_n)\in {\cal D}_\alpha^ F(\beta)$ for any given $L'>L$ provided that $n$ is large. So there exists $C>0$ which depends on $L$, such that
$y_n=\phi_L(x_n)\in $ is $(C, \eta, T)$-$\psi^*$-contracted. By Lemma \ref{Lem:infinitePliss}, there exists $L'>L$,
such that some points $z_n = \phi_{\tilde{L}}(x_n)$ $(L< \tilde{L}< L')$ is $(1, {\eta\over 2}, T)$-$\psi^*$-contracted.
By Lemma \ref{Lem:uniformestimate} and our choice of $\delta$ above, whenever $n$ is large enough,
$\exp_{z_n}{\cal N}_{z_n}(\delta|X(z_n)|)$ is in the stable domain of the sectional Poincar\'e map $P_t$ for any $t\ge 0$. By Lemma \ref{Lem:smallscale},
we have $\exp_{z_n}({\cal N}_{z_n}(\delta|X(z_n)|))\cap W^F(\sigma)\neq\emptyset$. In other words, in $n$ is large enough, the basin of $\gamma_n$ intersects the $W^{F}(\sigma)$. But $W^{F}(\sigma)$ contains only two orbits. The two orbits can only go to at most two sinks by forward iterations. This gives a contradiction.

\paragraph{Now we will the consider the case that $F$ is not expanding.}

By the domination, $E=E^s$ is uniformly contracted. Thus the stable manifold $W^s(\sigma)$ separates a small neighborhood of $\sigma$ into two parts: the upper part which contains $W^{F,+}(\sigma)$ and the lower part which contains $W^{F,-}(\sigma)$. By taking a subsequence if necessary, we assume that $\{x_n\}$ accumulates $\sigma$ in the upper part.

As in \cite{PuS09}, there are two cases for $W^{F,+}(\sigma)$:

\begin{enumerate}

\item $W^{F,+}(\sigma)$ is Lyapunov stable in the following sense: for any $\varepsilon>0$, there is $\delta>0$ such that $\varphi_t(W_\delta^{F,+}(\sigma))\subset W_\varepsilon^{F,+}(\sigma)$ for any $t>0$.

\item $W^{F,+}(\sigma)$ is contained in the unstable manifold of $\sigma$.

\end{enumerate}

In Case 1, the upper half part of the small neighborhood will be foliated by a strong stable foliation. This means that some point in $W^{F,+}$ is contained in the stable manifold of some $\gamma_n$. This contradicts to the local invariance of $W^{F,+}(\sigma)$.

For Case 2, by Lemma~\ref{Lem:uniquelyintegrable}, $W^{F,+}$ is uniquely defined. The argument will be similar to the case when $F$ is expanding. For completeness, we repeat it again.

Let $\delta =\delta(1, {\eta/2}, T)$ be the number given by Lemma \ref{Lem:uniformestimate}. By Lemma \ref{Lem:smallscale},
for this $\delta$, there exist $\alpha>0$ and $\beta>0$ such that for any $y\in {\cal D}_\alpha^ F (\beta)$, one has $\exp_{y}({\cal N}_{y}(\delta|X(y)|))\cap W^F(\sigma)\neq\emptyset$.
Lemma \ref{Lem:insidecone} gives the number $L=L(\alpha, \beta)>0$.
$\phi_{[L,L']}(x_n)\in {\cal D}_\alpha^ F(\beta)$ for any given $L'>L$ provided that $n$ is large. So there exists $C>0$ which depends on $L$, such that
$y_n=\phi_L(x_n)\in $ is $(C, \eta, T)$-$\psi^*$-contracted. By Lemma \ref{Lem:infinitePliss}, there exists $L'>L$,
such that some points $z_n = \phi_{\tilde{L}}(x_n)$ $(L< \tilde{L}< L')$ is $(1, {\eta\over 2}, T)$-$\psi^*$-contracted.
By Lemma \ref{Lem:uniformestimate} and our choice of $\delta$ above, whenever $n$ is large enough,
$\exp_{z_n}{\cal N}_{z_n}(\delta|X(z_n)|)$ is in the stable domain of the sectional Poincar\'e map $P_t$ for any $t\ge 0$. By Lemma \ref{Lem:smallscale},
we have $\exp_{z_n}({\cal N}_{z_n}(\delta|X(z_n)|))\cap W^{F,+}(\sigma)\neq\emptyset$. In other words, if $n$ is large enough, the basin of $\gamma_n$ intersects $W^{F,+}(\sigma)$. $W^{F,+}(\sigma)$ is just one orbit and uniquely defined. Thus it can only go to at most one sink by forward iterations. This gives a contradiction.
The proof of Theorem \ref{Thm:main} is complete.

\begin{Acknowledgements}
We would like to  thank Shaobo Gan for the knowledge from him about singular flows, especially on Liao's work.

\end{Acknowledgements}

\vskip 5pt

\noindent Dawei Yang

\noindent School of Mathematics, Jilin University, Changchun, 130012, P.R. China

\noindent yangdw1981@gmail.com

\vskip 5pt

\noindent Yong Zhang

\noindent Department of Mathematics, Suzhou University, Suzhou, 215006, P.R. China

\noindent yongzhang@suda.edu.cn


\begin{thebibliography}{1d}

\bibitem{GaY12}S. Gan and D. Yang, Morse-Smale systems and non-trivial horseshoes for three-dimensional singular flows, Arxiv:1302.0946v1.


\bibitem{HPS77} M. W. Hirsch; C. C. Pugh; M. Shub, Invariant manifolds, {\it Lecture
Notes in Mathematics}, {\bf 583} Springer-Verlag 1977.

\bibitem{Lia89}S. Liao, On $(\eta,d)$-contractible orbits of vector
fields, {\it Systems Science and Mathematical Sciences},
{\bf2}(1989), 193-227.



\bibitem{Pli72}V. Pliss, A hypothesis due to Smale, {\it Diff. Eq.} {\bf
8} (1972), 203-214.



\bibitem{PuS00} E. Pujals and M. Sambarino, Homoclinic tangencies and
hyperbolicity for surface diffeomorphisms, {\it Annals of Math.},
{\bf 151} (2000), 961-1023.


\bibitem{PuS09} E. Pujals and M. Sambarino, On the dynamics of dominated splitting. {\it Annals of Math.}, {\bf 169}(2009), 675-739.





\end{thebibliography}
\end{document}